\documentclass[12pt]{amsart}
\usepackage{amsmath}
\usepackage{amssymb}
\usepackage{latexsym}
\usepackage{amscd}
\usepackage{citesort}
\usepackage{graphicx} 
\usepackage{amsthm}
\usepackage{mathrsfs}
\usepackage{xypic}
\usepackage{bm}

\newdimen\AAdi%
\newbox\AAbo%
\def\AAk#1#2{\s_etbox\AAbo=\hbox{#2}\AAdi=\wd\AAbo\kern#1\AAdi{}}%
\def\AAr#1#2#3{\s_etbox\AAbo=\hbox{#2}\AAdi=\ht\AAbo\raise#1\AAdi\hbox{#3}}%
\font\tenmsb=msbm10 at 12pt \font\sevenmsb=msbm7 at 8pt
\font\fivemsb=msbm5 at 6pt
\newfam\msbfam
\textfont\msbfam=\tenmsb \scriptfont\msbfam=\sevenmsb
\scriptscriptfont\msbfam=\fivemsb
\def\Bbb#1{{\tenmsb\fam\msbfam#1}}
\newtheorem{thm}{Theorem}[section]

\newtheorem{cor}{Corollary}[section]
\newtheorem{rem}{Remark}[section]
\newtheorem{pro}{Proposition}[section]

\newcommand{\ba}{\begin{array}}
\newcommand{\ea}{\end{array}}

\newcommand{\Section}[2]{\setcounter{equation}{0}
\allowdisplaybreaks
\section[#1]{#2}}

\def\n{\nabla}

\def\ir#1{\mathbb R^{#1}}

\def\f#1#2{\frac{#1}{#2}}

\def\mc#1{\mathcal{#1}}

\def\a{\alpha}
\def\be{\beta}

\def\p#1{\partial #1}

\def\de{\delta}
\def\De{\Delta}

\def\ep{\varepsilon}

\def\g{\gamma}

\def\la{\lambda}
\def\La{\Lambda}
\def\om{\omega}
\def\Om{\Omega}
\def\th{\theta}
\def\Th{\Theta}

\def\w{\wedge}

\def\Hess{\mbox{Hess}}
\def\R{\Bbb{R}}

\def\lan{\langle}
\def\ran{\rangle}
\def\ra{\rightarrow}

\def\aint#1{-\hskip -4.5mm\int_{#1}}
\def\V{\mbox{Vol}}
\def\ol{\overline}

\subjclass{58E20,53A10.}

\begin{document}
\title
[Curvature estimates for minimal hypersurfaces] {Curvature estimates
for minimal hypersurfaces via generalized longitude function}
\author
[Ling Yang]{Ling Yang}
\address{School of Mathematical Sciences, Fudan University,
Shanghai 200433, China.} \email{yanglingfd@fudan.edu.cn}
\thanks{The author is partially supported by NSFC. He is grateful to the Max Planck
Institute for Mathematics in the Sciences in Leipzig for its
hospitality and  continuous support. }
\date{}

\begin{abstract}

On some specified convex supporting sets of spheres, we find a generalized longitude
function $\th$ whose level sets are totally geodesic. Given an
arbitrary (weakly) harmonic map $u$ into
spheres, the composition of $\th$ and $u$ satisfies an elliptic
equation of divergence type. With the aid of corresponding Harnack
inequality, we establish image shrinking property and then the
regularity results are followed. Applying such results to study the
Gauss image of minimal hypersurfaces in Euclidean spaces,
we obtain curvature estimates and corresponding Bernstein theorems.

\end{abstract}

\maketitle

\Section{Introduction}{Introduction}

The convexity plays an important role in regularity theory of harmonic maps. It is well-known that an open hemisphere is
the maximal geodesic ball in Euclidean sphere, and Hildebrandt-Kaul-Widman \cite{h-k-w} proved the regularity of harmonic
maps whose image is contained in a compact subset of an open hemisphere. One of the key points of the paper is the fact
that the composition of (weakly) harmonic map and a convex function on the target manifold gives a (weakly) subharmonic
function. One can then exploits the maximal principle for partial differential equations of elliptic type or in more refined
schemes, Moser's Harnack inequality.

Also using Moser's Harnack inequality, Moser \cite{m} obtained weakly Bernstein theorem for minimal hypersurfaces, which says
that an entire minimal graph $M=\{(x,f(x)):x\in \R^n\}$ has to be affine linear provided that the slope of the function $f$
is uniformly bounded. In the viewpoint of Gauss maps, the assumption on $|Df|$ equals to say that the Gauss image of $M$
is contained in a closed subset of an open hemisphere.

In \cite{jxy}, it is shown that one can do substantially better. More precisely, a weakly harmonic map into $S^n$
has to be regular whenever its image is contained in a compact subset of the complement of half of a equator (i.e. an upper
hemisphere of codimension 1), which contains the upper hemisphere and the lower hemisphere. In conjunction with Ruh-Vilms
theorem \cite{r-v}, one can prove the following Bernstein type theorem: Let $M$ be a complete imbedded minimal hypersurface
in $\R^{n+1}$, satisfying so-called DVP-condition, if the Gauss image of $M$ omits a neighborhood of half of a equator,
then $M$ has to be affine linear. It improves Moser's theorem. $\Bbb{V}:=S^n\backslash \ol{S}_+^{n-1}$ is a maximal \textit{convex
supporting set}, i.e. every compact set $K\subset \Bbb{V}$ submits a strictly convex function (see \cite{G}),
and $\Bbb{V}$ is maximal because as soon we enlarge $\Bbb{V}$, it will contain a closed geodesic. To make analytic
technologies applicable, one have to construct a smooth family of strictly convex functions on $K\subset \Bbb{V}$,
not just a single convex function. The construction is quite subtle. Based on these functions, one can use the Green
function test technique \cite{h-k-w}\cite{h-j-w}, telescoping trick of \cite{g-g}\cite{g-h} and image shrinking
method employing in \cite{h-k-w}\cite{h-j-w}\cite{g-j} to derive regularity theorem and Bernstein type result. In
the process, Moser's Harnack inequality \cite{m}\cite{b-g} and a-prior estimates for Green functions
\cite{g-w}\cite{bm} play a fundamental role.

It is natural for us to raise the following 2 questions.

Firstly, is $S^n\backslash \ol{S}_+^{n-1}$ the unique maximal convex supporting
set in $S^n$? If not, given a convex supporting set $\Bbb{V}\subset S^n$ which
is not contained in $S^n\backslash \ol{S}_+^{n-1}$, can we derive regularity results
(or Bernstein type results) when the image under harmonic map (or Gauss map, respectively) is contained
in a compact subset of $\Bbb{V}$?

Moser's theorem has been improved by Ecker-Huisken \cite{e-h}, which says that any entire minimal
graph has to be affine linear whenever $|Df|=o(\sqrt{|x|^2+f^2})$. In other words,
one can obtain Bernstein type theorem for minimal hypersurface $M$ whose Gauss image lies in an open
hemisphere; the image of $y\in M$ under Gauss map is allowed to tend to the equator (the boundary of hemisphere) when
$y$ diverges to the infinity in a controlled manner. Similarly, under the fundamental assumption
that $\g(M)\subset S^n\backslash \ol{S}_+^{n-1}$, where $\g$ denotes the Gauss map, can we derive
Bernstein type results by imposing an additional condition on the rate of convergence of $\g(y)$
to the boundary of $S^n\backslash \ol{S}_+^{n-1}$ as $y\ra \infty$? It is our second question.

We partially answer above 2 questions in the present paper. But our technique is a bit different from \cite{jxy}.
The function $\th$ on spheres, so-called \textit{generalized longitude function}, play a crucial role in our statement.

Let $\pi$ be the natural projection from $\R^{n+1}$ onto $\R^2$, then $\pi$ maps $S^n$ onto $\overline{\Bbb{D}}$, the 2-dimensional closed
unit disk. It is easily-seen that the preimage of $(0,0)$ under $\pi$ is the subsphere of codimension 2,
which is denoted by $S^{n-2}$. Given $\Bbb{V}\subset S^n$, once $\Bbb{V}$ is a simply-connected subset of $S^n\backslash S^{n-2}$,
the composition of $\pi$ and angular coordinate of $\overline{\Bbb{D}}\backslash \{(0,0)\}$ yields
a real-valued function, denoted by $\th$. When $n=2$, $\th$ becomes the longitude function,
so $\th$ is called \textit{generalized longitude function}. Each level set of $\th$ is contained in a hemisphere of codimension 1, which is
totally geodesic. It is not hard for us to calculate $\Hess\ \th$ and moreover, prove $\Bbb{V}$ is a convex-supporting set whenever
$\Bbb{V}$ is a simply-connected subset of $S^n\backslash S^{n-2}$ (see Proposition \ref{p2}). It means
$S^n\backslash \ol{S}_{+}^{n-1}$ is not the unique maximal convex supporting set and the first question is partially answered. But
it is still unknown what is the sufficient and necessary condition ensuring $\Bbb{V}$ be a convex supporting set of $S^n$.

Using composition formula, we can deduce the partial differential equation (\ref{la}) that $\th\circ u$ satisfies whenever $u$ is a (weakly) harmonic
map into $\Bbb{V}\subset S^n$. We note that (\ref{la}) can also be derived in the framework of warped product structure,
see \cite{so}. Following the idea of Moser \cite{m} and \cite{b-g}, one can derive Harnack's inequalities for $\th\circ u$
when $M$ satisfies so-called \textit{local DSVP-condition} with respect to a fixed point $y_0\in M$. Here 'D' represents the existence
of a distance function $d$, the metric topology induced by which coincides with the initial topology; 'V' denotes the condition
on the volume growth of metric balls centered at $y_0$ as a function of their radius; 'S' and 'P' are respectively Sobolev type inequalities
and Neumann-Poincar\'{e} inequalies for functions defined on metric balls centered at $y_0$ with uniform constants. It is easy
to show local DSVP-condition is weaker than DVP-condition in \cite{jxy}. Harnack's inequalities implies image shrinking property,
and it follows a regularity theorem of weakly harmonic maps into sphere with image restrictions (see Theorem \ref{t2}), which not
only generalize but also improve the regularity theorem in \cite{jxy}.

Finally, in conjunction with image shrinking property and Ecker-Huisken's curvature estimates \cite{e-h} for minimal graphs, we deduce
curvature estimates for minimal hypersurfacs with Gauss image restrictions, which implies a Bernstein type theorem as follows.

\begin{thm}\label{be}
Let $M^m\subset \R^{m+1}$ be an imbedded complete minimal hypersurface with Euclidean volume growth. There is $y_0\in M$, such that
the following Neumann-Poincar\'{e} inequality
$$\int_{B_R(y_0)}|v-\bar{v}_R|^2*1\leq CR^2\int_{B_R(y_0)}|\n v|^2*1\qquad \forall v\in C^\infty(B_R(y_0))$$
holds with a positive constant $C$ not depending on $R$, where $B_R(y_0)$ denotes the extrinsic ball centered at $y_0$ and of radius
$R$ and $\bar{v}_R$ is the average values of $v$ on $B_R(y_0)$. If the Gauss image of $M$ is contained in $S^m\backslash \ol{S}_+^{m-1}$,
and
$$\sup_{B_R(y_0)}d(\cdot,S^{m-2})^{-1}\circ \g=o(\log\log R)$$
then $M$ has to be an affine linear subspace.
\end{thm}

It is comparable with Theorem 6.5 in \cite{jxy}. Firstly, the previous one requires Neumann-Poincar\'{e} inequality holds true
for every extrinsic balls centered at each point of $M$ and of arbitrary radius with a uniform constant; so the assumption of
Theorem \ref{be} on $M$ is weaker.

Please note that $\p(S^m\backslash \ol{S}_+^{m-1})=S^{m-2}\cup A$ with $A$ the preimage of the interval $(0,1]$ under $\pi$. Given
a sequence $\{y_k:k\in \Bbb{Z}^+\}$ in $M$ tending to $\infty$, $\{\g(y_k)\}$ cannot converge to the boundary
of $S^m\backslash \ol{S}_+^{m-1}$
by the assumption of Theorem 6.5 in \cite{jxy}. In contrast, $\{\g(y_k)\}$ is allowed to converge to an arbitrary
point in $A$ at arbitrary speed, or any point of $S^{m-2}$ in a controlled manner. Hence Theorem \ref{be} partially answers the second
question that we have raised.

If we replace the Gauss image restriction on $M$ by assuming $\g(M)$ is contained in a closed, simply-connected subset
of $S^m\backslash S^{m-2}$, then again based on Gauss image shrinking property one can get the corresponding Bernstein type result,
it is generalization of Theorem 6.5 in \cite{jxy}.

As shown in \cite{m-s}\cite{b-g}, area-minimizing hypersurfaces satisfy local DSVP-condition, hence Theorem \ref{t5}
and Corollary \ref{cor2} are followed from Theorem \ref{be}. Unfortunately we do not know whether the Bernstein type results are optimal.

\bigskip\bigskip

\Section{Generalized longitude functions on spheres}{Generalized longitude functions on spheres}\label{s1}

There is a covering map $\chi:(-\f{\pi}{2},\f{\pi}{2})\times \R\ra S^2\backslash \{N,S\}$,
$$(\varphi,\th)\mapsto (\cos\varphi\cos\th,\cos\varphi\sin\th,\sin\varphi)$$
where $N$ and $S$ are the north pole and the south pole, $\varphi$ and $\th$ are latitude and longitude, respectively.
$\{\varphi,\th\}$ is called the geographic coordinate of $S^2$. Each level set of $\th$ is a meridian, i.e. a half of
great circle connecting the north pole and the south pole. Although $\chi$ is not one-to-one, the restriction
of $\chi$ on $(-\f{\pi}{2},\f{\pi}{2})\times (-\pi,\pi)$ is a bijection onto an open domain $\Bbb{V}$ that is obtained
by deleting the International date line from $S^2$. It has shown in \cite{G}\cite{jxy} that $\Bbb{V}$ is a maximal convex
supporting set of $S^2$. i.e. arbitrary compact set $K\subset \Bbb{V}$ submits a strictly convex function.

The longitude function $\th$ and open domain $\Bbb{V}$ can be generalized to higher dimensional cases.

Let $\pi$ be a natural projection from $\R^{n+1}$ onto $\R^2$, which maps $(x_1,\cdots,x_{n+1})$ to $(x_1,x_2)$, then
it is easily-seen that $\pi(S^n)=\overline{\Bbb{D}}$, the 2-dimensional closed unit disk. On it we shall use the
polar coordinate system; the radial coordinate and the angular coordinate are respectively denoted by $r$ and $\th$. In other words,
there exists a covering mapping $\chi: (0,1]\times \R\ra \overline{\Bbb{D}}\backslash \{(0,0)\}$
$$(r,\th)\mapsto (r\cos\th,r\sin\th).$$

Assume $\Bbb{V}$ is a simply connected subset of $S^n\backslash S^{n-2}=\pi^{-1}\big(\overline{\Bbb{D}}\backslash \{(0,0)\}\big)$,
then lifting theorem in homotopy theory enable us
to find a smooth mapping $\Psi:\Bbb{V}\ra (0,1]\times \R$
$$x\mapsto \Psi(x)=(r(x),\th(x))$$
 such that the following commutative diagram holds
$$\CD
 \Bbb{V} @>\Psi>> (0,1]\times \Bbb{R}  \\
 @V\mathbf{Id}VV     @VV\chi V \\
 \Bbb{V}  @>{\pi}>>\overline{\Bbb{D}}\backslash \{(0,0)\}
\endCD$$
In other words,
\begin{equation}\label{th}
(x_1,x_2)=\pi(x)=\chi\circ \Psi(x)=r(x)\big(\cos\th(x),\sin\th(x)\big)\qquad \forall x\in \Bbb{V}.
\end{equation}

For every fixed vector $a\in \R^{n+1}$, $(\cdot,a)$ is obviously a smooth function on $S^n$. Here and in the sequel
$(\cdot,\cdot)$ denotes the canonical Euclidean inner product. From the theory of spherical geometry,
the normal geodesic $\g$ starting from $x$ and with the initial vector $v$  ($|v|=1$ and $(x,v)=0$) has the form
$$\g(t)=\cos t\ x+\sin t\ v.$$
Then
$$(\g(t),a)=\cos t\ (x,a)+\sin t\ (v,a).$$
Differentiating twice both sides of the above equation with respect to $t$ implies
$$\Hess(\cdot,a)(v,v)=-(x,a).$$
In conjunction with the formula $2\Hess f(v,w)=\Hess f(v+w,v+w)-\Hess f(v,v)-\Hess f(w,w)$, it is easy to obtain
\begin{equation}\label{hess}
\Hess(\cdot,a)=-(\cdot,a)\ g_s
\end{equation}
where $g_s$ is the standard metric on $S^n$. Especially putting $a=\ep_i$ gives
\begin{equation}\label{hessx}
\Hess\ x_i=-x_i\ g_s\qquad \text{for every }1\leq i\leq n+1
\end{equation}
where $\ep_i$ is a unit vector in $\R^{n+1}$ whose $i$-th coordinate is 1 and other coordinates are all 0.

From (\ref{th}), $r^2=x_1^2+x_2^2$, hence
\begin{equation}\label{hessr1}\aligned
\Hess\ r^2=&2x_1\Hess\ x_1+2x_2\Hess\ x_2+2dx_1\otimes dx_1+2dx_2\otimes dx_2\\
          =&-2x_1^2\ g_s-2x_2^2\ g_s+2(\cos\th\ dr-r\sin\th\ d\th)\otimes(\cos\th\ dr-r\sin\th\ d\th)\\
           &+2(\sin\th\ dr+r\cos\th\ d\th)\otimes (\sin\th\ dr+r\cos\th\ d\th)\\
          =&-2r^2\ g_s+2dr\otimes dr+2r^2d\th\otimes d\th.
          \endaligned
\end{equation}
On the other hand,
\begin{equation}\label{hessr2}
\Hess\ r^2=2r\Hess\ r+2dr\otimes dr.
\end{equation}
(\ref{hessr1}) and (\ref{hessr2}) implies
\begin{equation}\label{hessr}
\Hess\ r=-r\ g_s+rd\th\otimes d\th.
\end{equation}
Furthermore (\ref{hessx}), (\ref{th}) and (\ref{hessr}) yield
$$\aligned
-x_1\ g_s&=\Hess\ x_1\\
       &=\cos\th\ \Hess\ r-r\sin\th\ \Hess\ \th-r\cos\th\ d\th\otimes d\th-\sin\th(dr\otimes d\th+d\th\otimes dr)\\
       &=-x_1\ g_s+x_1\ d\th\otimes d\th-r\sin\th\ \Hess\ \th-x_1\ d\th\otimes d\th-\sin\th(dr\otimes d\th+d\th\otimes dr)\\
       &=-x_1\ g_s-r\sin\th\ \Hess\ \th-\sin\th(dr\otimes d\th+d\th\otimes dr).
\endaligned$$
i.e.
$$r\sin\th\ \Hess\ \th=-\sin\th(dr\otimes d\th+d\th\otimes dr).$$
Similarly computing $\Hess\ x_2$ with the aid of (\ref{th}) and (\ref{hessr}) yields
$$r\cos\th\ \Hess\ \th=-\cos\th(dr\otimes d\th+d\th\otimes dr).$$
Therefore
\begin{equation}
\Hess\ \th=-r^{-1}(dr\otimes d\th+d\th\otimes dr).
\end{equation}
The above formula tells us $\Hess\ \th(v,v)=0$ for arbitrary $v\in T\Bbb{V}$ satisfying $\th(v)=0$; in other words,
the level sets of $\th$ are all totally geodesic hypersurfaces.

For arbitrary compact subset $K\subset \Bbb{V}$, there is a constant $c\in (0,1)$, such that $r>c$ on $K$. Therefore, the function
\begin{equation}
\phi=\th+\arcsin(cr^{-1})
\end{equation}
is well-defined on $K$. A direct calculation same as in \cite{jxy} shows $\Hess\ \phi(X,X)>0$ for every $X\in TK$ satisfying $|X|=1$
and $d\phi(X)=0$. Thereby Lemma 2.1 in \cite{jxy} enable us to get $\la$ large enough, such that
\begin{equation}
F=\la^{-1}\exp(\la\phi)
\end{equation}
is strictly convex on $K$. Hence we have

\begin{pro}\label{p2}

Any simply connected subset of $S^n\backslash S^{n-2}$ is a convex supporting set of $S^n$.

\end{pro}

Our denotation of $\ol{S}_+^{n-1}$ is same as in \cite{jxy}, which is a codimension 1 closed hemisphere in $S^n$. More precisely,
\begin{equation}
\ol{S}_+^{n-1}=\{(x_1,\cdots,x_{n+1})\in S^n: x_1\geq 0,x_2=0\}.
\end{equation}
Then $S^n\backslash \ol{S}_+^{n-1}$ is obviously a simply connected subset
of $S^n\backslash S^{n-2}$, and one can choose $\th$ to be a $(0,2\pi)$-valued
function in $S^n\backslash \ol{S}_+^{n-1}$ as in \cite{jxy}.  Thus the above proposition tells us $S^n\backslash \ol{S}_+^{n-1}$ is a
convex supporting set, which is also a maximal one, since if we add even a single point to $S^n\backslash \ol{S}_+^{n-1}$, it shall
contain a closed geodesic (see \cite{jxy}). Although
Proposition \ref{p2} generalizes the conclusion of Theorem 2.1 in \cite{jxy},
it is still unsloved what is the sufficient and necessary condition ensuring $\Bbb{V}$ to be a convex supporting set. To see the
relationship between convex supporting sets and closed geodesics, please have a look at Appendix.

Now we assume $M$ is an $m$-dimensional Riemannian manifold, $\Bbb{V}$ is a simply connected subset of $S^n\backslash S^{n-2}$.
If $u: M\ra \Bbb{V}$ is a harmonic map, then the composition function $\th\circ u$ defines
a smooth function on $M$. Using composition formula, we have
\begin{equation}\label{la1}\aligned
\De(\th\circ u)&=\Hess\ \th(u_*e_\a,u_*e_\a)+d\th\big(\tau(u))\\
               &=-2(r^{-1}\circ u)dr(u_*e_\a)d\th(u_*e_\a)\\
               &=-2(r^{-1}\circ u)\big\lan \n(r\circ u), \n(\th\circ u)\big\ran
               \endaligned
\end{equation}
where $\tau$ denotes the tensor field of $u$, which is identically zero when $u$ is harmonic; $\n$ denotes the Levi-Civita
connection on $M$; and $\lan\ ,\ \ran$ is the Riemannian metric on $M$. Here and in the sequel we denote by
$\{e_1,\cdots,e_m\}$ a local orthonormal frame field on $M$. We use the summation convention and assume the range
of indices
$$1\leq \a\leq m.$$
(\ref{la1}) is equivalent to
$$\De (\th\circ u)+2(r^{-1}\circ u)\big\lan \n (r\circ u), \n (\th\circ u)\big\ran=0.$$
Multiplying both sides by $r^2\circ u$ yields
\begin{equation}\label{la}
\text{div}\big((r^2\circ u)\n(\th\circ u)\big)=0.
\end{equation}
Here $\text{div}$ is divergence operation with respect to the metric on $M$.

Please note that if we weaken the condition on $u$ by just assuming it is a weakly harmonic map, then a direct computation similar to \cite{J} \S 8.5
and \cite{jxy} shows that $\th\circ u$ is a weak solution to the partial differential equation (\ref{la}). More precisely,
for arbitrary smooth function $\phi$ on $M$ with compact supporting set,
\begin{equation}\label{div}
\int_M (r^2\circ u) \big\lan  \n\phi,\n(\th\circ u)\big\ran*1=0.
\end{equation}
It also can be derived from (\ref{la}) with the aid of classical divergence theorem when $u$ is harmonic.

(\ref{la}) and (\ref{div}) shall play an important part in the next sections.

\begin{rem}

Solomon \cite{so} showed that $S^n\backslash S^{n-2}$ has so-called warped product structure. More precisely,
if we denote
$$S^{n-1}_+=\{(y_1,\cdots,y_n)\in S^{n-1}:y_1>0\},$$
then there is a diffeomorphism $F$ from $S^{n-1}_+\times S^1$ to $S^n\backslash S^{n-2}$
$$F\big((y_1,\cdots,y_n),\varphi\big)=(y_1\cos\varphi,y_1\sin\varphi,y_2,\cdots,y_n).$$
From the viewpoint, $\varphi$ can be regarded as a smooth $S^1$-valued function on $S^n\backslash S^{n-2}$,
and the level sets of $\varphi$ are all totally geodesic and orthogonal to $\f{\p}{\p \varphi}$.

Please note that the function $\th$ we have defined can be seen as the lift of the restriction of $\varphi$ on a convex supporting set, e.g. $S^n\backslash \ol{S}^{n-1}_+$.
Hence (\ref{la}) can also be derived from Lemma 1 in \cite{so}.

\end{rem}

\bigskip

\Section{Harnack's inequalities for elliptic differential equations}{Harnack's inequalities for elliptic differential equations}
\label{s2}

Let $(M,g)$ be a Riemannian manifold, $A$ be a section of vector bundle $T^*M\otimes TM$, such that
for every $y\in M$ and nonzero $X,Y\in T_y M$,
\begin{equation}
\lan X,A(Y)\ran=\lan Y,A(X)\ran
\end{equation}
and
\begin{equation}
\lan X,A(X)\ran>0.
\end{equation}
Then we call $A$ is symmetric and positive definite, and
\begin{equation}\label{eq}
\text{div}\big(A(\n f)\big)=0
\end{equation}
is obviously a partial differential equation of elliptic type.

At first, we assume there is a distance function $d$ on $M$, and the metric
topology induced by $d$ is equivalent to the initial topology of
$M$; moreover, for each $y_1,y_2\in M$, $d(y_1,y_2)\leq
\rho(y_1,y_2)$, where $\rho(\cdot,\cdot)$ is the distance function of
$M$ with respect to the Riemannian metric.
Obviously, $\rho$ is one of the required functions.

Now we fix $y_0\in M$, and let $B_R=B_R(y_0)$ be the ball centered at $y_0$ of radius $R$ given by
the distance function $d$. We assume every function in $B_R$ of $H_0^{1,2}$ type is also a
$L^{2\nu}$-function with $\nu>1$, and there is a positive constant $K_1$, such that for every $r\in [\f{R}{2},R]$
and $v\in H_0^{1,2}(B_r)$,
\begin{equation}\label{con2}
\Big(\aint{B_r}|v|^{2\nu}\Big)^{\f{1}{2\nu}}\leq
K_1r \big(\aint{B_r}|\n v|^2\big)^{\f{1}{2}}.
\end{equation}
This is a Sobolev type inequality. Here $-\hskip -4mm\int_{\Om}v$ denotes the average value of $v$ on arbitrary domain $\Om\subset M$, i.e.
\begin{equation*}
\aint{\Om}v=\f{\int_\Om v*1}{\V(\Om)}.
\end{equation*}
 And $K_2,K_3,\la,\mu,L$ are positive constants satisfying
\begin{eqnarray}
&&\text{Vol}(B_R)\leq K_2{\text{Vol}(B_{\f{R}{2}})},\label{double}\\
&&\sup_{v\neq 0,\int_{B_{\f{3R}{4}}}v*1=0}\f{\int_{B_{\f{3R}{4}}} v^2*1}{\int_{B_{\f{3R}{4}}}|\n v|^2*1}\leq K_3R^2,\label{poin2}\\
&&\la_1:=\inf_{X\in TB_R,X\neq 0}\f{\lan X,A(X)\ran}{\lan X,X\ran},\label{eig1}\\
&&\la_2:=\sup_{X\in TB_R,X\neq 0}\f{\lan X,A(X)\ran}{\lan X,X\ran},\label{eig2}\\
&&L:=\f{\la_2}{\la_1}\label{eig}.
\end{eqnarray}
By classical spectrum theory of harmonic operators, if we denote by $\mu_2$ the second
eigenvalue of $\De v+\mu v=0$ in $B_{\f{3R}{4}}$, where $v$ has a vanish normal derivative on the
boundary of $B_{\f{3R}{4}}$, then the left hand side of (\ref{poin2}) equals $\mu_2^{-1}$. For this reason, an equivalent form of (\ref{poin2})
\begin{equation}\label{str-p}
\int_{B_{\f{ 3R}{4}}}|v-\bar{v}_{B_{ \f{3R}{4}}}|^2*1\leq
K_3R^2\int_{B_{\f{3R}{4}}}|\n v|^2 *1
\end{equation}
is called Neumann-Poincar\'{e} inequality in some references, e.g. \cite{cm}.

Following the idea of \cite{m} and \cite{b-g}, it is not hard for us to derive Harnack inequalities as follows.

\begin{pro}\label{p1}
Let $(M,g)$ be a Riemannian manifold equipped with a distance function $d$, the metric topology induced
by $d$ coincides with initial topology of $M$, and $d(\cdot,\cdot)\leq \rho(\cdot,\cdot)$. If $f$ is a positive
(weak) solution to (\ref{eq}) on the metric ball centered at $y_0$ and of radius $R$, then there exists a positive constant $C_0$, only depending on
$K_1,K_2,K_3$ and $\nu$, but not on $f$, $L$ and $R$, such that
\begin{equation}
\log f_{+,\f{R}{2}}-\log f_{-,\f{R}{2}}\leq C_0L^{\f{1}{2}}.
\end{equation}
\end{pro}

Here and in the sequel, $B_R=B_R(y_0)$ and
\begin{equation}
f_{+,R}:=\sup_{B_R}f,\qquad f_{-,R}:=\inf_{B_R}f.
\end{equation}

\begin{rem}

For divergence elliptic partial differential equations in open domain $\Om\subset \R^n$, one can deduce the classical Harnack inequality
(see \cite{m})
\begin{equation}\label{har15}
\sup_{\Om'}f\leq C(\Om',\Om)^{L^{\f{1}{2}}}\inf_{\Om'}f
\end{equation}
for arbitrary $\Om'\subset\subset \Om$. In addition $\Om$ is convex one can take
\begin{equation}
C(\Om',\Om)=\Big(\f{\text{diam}\ \Om}{\text{dist}(\Om',\p \Om)}\Big)^{\be}
\end{equation}
with a positive constant $\be$. Bombieri-Giusti \cite{b-g} generalized the conclusion to this type of partial differential equations
on area-minimizing hypersurfaces on Euclidean spaces. The above proposition is a further generalization. Please note that the example
$$\f{\p^2 f}{\p x^2}+L\f{\p^2 f}{\p y^2}=0$$
with solution
$$f=\exp(L^{\f{1}{2}}x)\cos y$$
shows that the dependence on $L$ in (\ref{har15}) cannot be improved.

\end{rem}

Now we let $f$ be an arbitrary (weak) $L^\infty$-solution of (\ref{eq}) on $B_R$ (not necessarily positive), then $f-f_{-,R}+\ep$ is obviously a positive
(weak) $L^\infty$-solution
for each $\ep>0$. Applying Proposition \ref{p1} to $f-f_{-,R}+\ep$ yields
$$\log(f_{+,\f{R}{2}}-f_{-,R}+\ep)-\log(f_{-,\f{R}{2}}-f_{-,R}+\ep)\leq C_0L^{\f{1}{2}},$$
i.e.
$$f_{+,\f{R}{2}}-f_{-,R}+\ep\leq \exp(C_0L^{\f{1}{2}})(f_{-,\f{R}{2}}-f_{-,R}+\ep).$$
Letting $\ep\ra 0$ implies
$$f_{+,\f{R}{2}}-f_{-,R}\leq \exp(C_0L^{\f{1}{2}})(f_{-,\f{R}{2}}-f_{-,R}),$$
then
$$\aligned
&f_{+,\f{R}{2}}-f_{-,\f{R}{2}}=(f_{+,\f{R}{2}}-f_{-,R})-(f_{-,\f{R}{2}}-f_{-,R})\\
\leq&\big(1-\exp(-C_0L^{\f{1}{2}})\big)(f_{+,\f{R}{2}}-f_{-,R})\leq \big(1-\exp(-C_0L^{\f{1}{2}})\big)(f_{+,R}-f_{-,R}).
\endaligned$$
Thereby we get the estimate for oscillation of $f$ as follows.

\begin{cor}\label{cor1}
Our assumption on $M$ is same as in Proposition \ref{p1}.
If $f$ is a (weak) $L^\infty$-solution
of (\ref{eq}) in $B_R=B_R(y_0)$ with $y_0\in M$, then the oscillation of $f$ on $B_{\f{R}{2}}$ could be estimated by
\begin{equation}
\text{osc}_{B_{\f{R}{2}}}f\leq \big(1-\exp(-C_0L^{\f{1}{2}})\big)\text{osc}_{B_R}f
\end{equation}
with a positive constant $C_0$ depending on $K_1,K_2,K_3$ and $\nu$, but not on $L$ and $R$.
\end{cor}

\bigskip

\Section{Image shrinking property of harmonic maps}{Image shrinking property of harmonic maps}

Our denotation and assumption is same as in Section \ref{s1} and Section \ref{s2}.
Compared with (\ref{la}) and (\ref{eq}),
$\th\circ u$  satisfies a type of elliptic Partial differential equation with $A=(r^2\circ u)\mathbf{Id}$, where
$\mathbf{Id}$ denotes a smooth section of $T^*M\otimes TM$ satisfying
$$\mathbf{Id}(X)=X\qquad \text{for every }X\in TM.$$
Since $r$ is a $(0,1]$-valued function,
\begin{equation}
L:=\f{\sup_{B_R}(r^2\circ u)}{\inf_{B_R}(r^2\circ u)}\leq \sup_{B_R}(r^{-2}\circ u).
\end{equation}

Now we give an additional assumption on $M$
that there is $R_0\in (0,+\infty]$, one can find uniform constants $K_1,K_2,K_3$
and $\nu$ which are all independent of $R\in (0,R_0]$, such that the estimates
(\ref{con2})-(\ref{poin2}) hold true. As a matter of convenience, we call $M$
satisfies 'local DSVP-condition' with respect to $y_0$ in the sequel.

\begin{rem}\label{r1}

Local DSVP-condition is comparable with DVP-condition in \cite{jxy}. From the work of Saloff-Coste \cite{sc} and Biroli-Mosco \cite{bm},
if 'doubling property' (\ref{double}) holds for arbitrary $y\in M$, $R\leq R_0$, and
Neumann-Poincar\'{e} inequality (\ref{str-p}) holds for arbitrary $B_{\f{3R}{4}}(y)\subset\subset M$ with uniform constants $R_0$, $K_2$ and $K_3$,
then Sobolev inequality (\ref{con2}) is satisfied whenever $B_{2R}(y)\subset\subset M$ and $R\leq \f{R_0}{2}$, with constants
$\nu$ and $K_1$ depending only on $K_2$ and $K_3$. it means that DVP-condition implies
local DSVP-condition with respect to arbitrary $y\in M$.

\end{rem}

  Corollary \ref{cor1} gives
\begin{equation}\label{osc1}
\text{osc}_{B_{\f{R}{2}}}(\th\circ u)\leq \big(1-\exp(-C_0\sup_{B_R}(r^{-1}\circ u))\big)\text{osc}_{B_R}(\th\circ u)
\end{equation}
with a constant $C_0$ independent of $R\leq R_0$. As a matter of convenience, we shall use abbreviations as follows
\begin{equation}
\Th:=\th\circ u,\qquad M(R):=\sup_{B_R}(r^{-1}\circ u)\ (R\in (0,R_0])
\end{equation}
in the sequel. Taking logarithms of both sides of (\ref{osc1}) gives
\begin{equation}
\log \text{osc}_{B_{\f{R}{2}}}\Th-\log\text{osc}_{B_R}\Th\leq \log\big(1-\exp(-C_0M(R))\big)
\end{equation}
for arbitrary $R\leq R_0$. After iteration we arrive at
\begin{equation}\label{osc2}
\log\text{osc}_{B_{2^{-k}R_0}}\Th-\log\text{osc}_{B_{R_0}}\Th\leq \sum_{j=0}^{k-1}\log\big(1-\exp(-C_0M(2^{-j}R_0))\big)
\end{equation}
for every $k\in \Bbb{Z}^+$.
By additionally defining $M(R)=M(R_0)$ when $R\in [R_0,2R_0]$, $M$ can be regarded as an increasing function on
$(0,2R_0]$. The right hand of above inequality could be estimated by
\begin{equation}\label{osc3}\aligned
&\sum_{j=0}^{k-1}\log\big(1-\exp(-C_0M(2^{-j}R_0))\big)\\
\leq&\int_{-1}^{k-1}\log\big(1-\exp(-C_0M(2^{-t}R_0))\big)dt\\
\leq&(\log 2)^{-1}\int_{2^{-k+1}R_0}^{2R_0}R^{-1}\log\big(1-\exp(-C_0M(R))\big)dR
\endaligned
\end{equation}
For every $R\leq \f{R_0}{2}$, there exists $k\in \Bbb{Z}^+$, such that $2^{-k-1}R_0<R\leq 2^{-k}R_0$; it is easy to get the following
estimate of oscillation by combining (\ref{osc2}) and (\ref{osc3}):
\begin{equation}\label{sh4}
\log\text{osc}_{B_R}\Th-\log\text{osc}_{B_{R_0}}\Th\leq (\log 2)^{-1}\int_{4R}^{2R_0}R^{-1}\log\big(1-\exp(-C_0M(R))\big)dR.
\end{equation}
Consider function $t\in (0,\exp(-C_0)]\mapsto -\f{\log(1-t)}{t}$;
since
$$\lim_{t\ra 0^+}-\f{\log(1-t)}{t}=\lim_{t\ra 0^+}-\f{\big[\log(1-t)\big]'}{t'}=1,$$
there is a positive constant $c_1$, depending only on $C_0$, such that
\begin{equation}\label{c8}
-\f{\log(1-t)}{t}\geq c_1\qquad \text{i.e. }\log(1-t)\leq -c_1t
\end{equation}
for all $t\in (0,\exp(-C_0)]$. Especially
$$\log\big(1-\exp(-C_0M(R))\big)\leq -c_1\exp(-C_0M(R)).$$
Substituting it into (\ref{sh4}) gives
\begin{equation}\label{sh2}
\log\text{osc}_{B_R}\Th-\log\text{osc}_{B_{R_0}}\Th\leq -(\log 2)^{-1}c_1\int_{4R}^{2R_0}R^{-1}\exp(-C_0M(R))dR.
\end{equation}
Again using the monotonicity of $M(R)$ implies
\begin{equation}\label{sh3}\aligned
\text{the right hand side of (\ref{sh2})}
\leq&-(\log 2)^{-1}c_1\exp(-C_0M(R_0))\int_{4R}^{2R_0}R^{-1}dR\\
=&-(\log 2)^{-1}c_1\exp(-C_0M(R_0))\log\Big(\f{R_0}{2R}\Big).
\endaligned
\end{equation}
From the estimates we can get so-called 'image shrinking property' of (weakly) harmonic maps.

\begin{thm}\label{t1}

Let $(M,g)$ be a Riemannian manifold satisfying local DSVP-condition with respect to $y_0\in M$ with constants $R_0>0, \nu>1$ and
$K_1,K_2,K_3>0$, $\Bbb{V}$ be a simply connected subset of $S^n\backslash S^{n-2}$.
 If $u: M\ra S^n$ is a (weakly) harmonic map, $u(B_{R_0})\subset K$, where $K$ is a compact subset of $\Bbb{V}$,
then there exists positive constants $C_0$ and $C_1$, depending only on $\nu,K_1,K_2,K_3$ and $K$, such that the image of $B_{R_1}$
under $u$ is contained in a closed geodesic ball of radius $\arccos\big(\f{1}{2}M(R_1)^{-1}\big)<\f{\pi}{2}$ in $S^n$, where
\begin{equation}\label{sh}
R_1:=\f{1}{2}\exp\big(-C_1\exp(C_0M(R_0))\big)R_0.
\end{equation}
Especially if $u(B_{R_0})\subset S^n\backslash\ol{S}_+^{n-1}$,
then our conclusion still holds true when we just assume
$$M(R_0):=\sup_{B_{R_0}}(r^{-1}\circ u)<+\infty.$$

\end{thm}

\begin{rem}

Theorem \ref{t1} is an improvement and a generalization of Theorem 5.1 in \cite{jxy}. Firstly, as shown in Remark \ref{r1}, local DSVP-condition can
be derived from DVP-condition, so our assumption on $M$ is weaker than that in \cite{jxy}. Secondly, in \cite{jxy}, 'image shrinking property' of
$u$ requires that $u(B_{R_0})$ is contained in a compact set $K\subset S^n\backslash \ol{S}_+^{n-1}$,
which is obviously a simply connected subset of $S^n\backslash S^{n-2}$, therefore the first statement generalizes the corresponding
 conclusion in \cite{jxy}.
Finally, if  $u(B_{R_0})\subset S^n\backslash \ol{S}_+^{n-1}$, then image shrinking property still hold true when just assuming
the composition of $r^{-1}$ and $u$ is bounded; in contrast, $\th\circ u$ can converge to $0$ or $2\pi$ at arbitrary speed near the boundary of $B_{R_0}$.
It is an improvement.

\end{rem}

\begin{proof}

Since $K$ is a closed subset of $\Bbb{V}$, $\Th=\th\circ u$ is a bounded function
on $B_{R_0}$. More precisely, there is a positive constant $c_2$ depending only on $K$, such that
\begin{equation}\label{osc4}
\text{osc}_{B_{R_0}}\Th\leq c_2.
\end{equation}

The constants $C_0$ and $c_1$ has been given above.
Now we choose
\begin{equation}
C_1:=\log 2\log \Big(\f{3c_2}{2\pi}\Big)c_1^{-1}
\end{equation}
then by (\ref{sh}),
\begin{equation}\label{sh1}
R_1=\f{1}{2}\exp\Big(-\log 2\log \Big(\f{3c_2}{2\pi}\Big)c_1^{-1}\exp(C_0M(R_0))\Big)R_0
\end{equation}
Substituting (\ref{sh1}) into (\ref{sh2}) and (\ref{sh3}) yields
$$\log\text{osc}_{B_{R_1}}\Th-\log\text{osc}_{B_{R_0}}\Th\leq -\log \Big(\f{3c_2}{2\pi}\Big).$$
In conjunction with (\ref{osc4}) we have  $\text{osc}_{B_{R_1}}\Th\leq \f{2\pi}{3}$. It enable us to find
$\th_0\in \R$, such that
$$(\th\circ u)\big|_{B_{R_1}}\leq [\th_0-\f{\pi}{3},\th_0+\f{\pi}{3}].$$
Denote $x_0=(\cos\th_0,\sin\th_0,0,\cdots,0)$, then for arbitrary $y\in B_{R_1}$,
$$\big(u(y),x_0\big)=r\circ u(y)\big(\cos(\Th-\th_0)\big)\geq \f{1}{2}r\circ u(y)\geq \f{1}{2}M(R_1)^{-1}$$
which implies $u(B_{R_1})$ is contained in the closed geodesic ball centered at $x_0$ and of radius $\arccos\big(\f{1}{2}M(R_1)^{-1}\big)$.

Noting that condition $u(B_{R_0})\subset S^n\backslash \ol{S}_{+}^{n-1}$ implies $\text{osc}_{B_{R_0}}\Th\leq 2\pi$, we can derive the second statement in the same way.
\end{proof}

Furthermore a regularity theorem of weakly harmonic maps easily follows.

\begin{thm}\label{t2}

Let $(M,g)$ be an arbitrary Riemannian manifold, $u: M\ra S^n$ be a weakly harmonic map, and
$\Bbb{V}$ be a simply connected subset of $S^n\backslash S^{n-2}$. Given $y_0\in M$, if there is
a neighborhood $U$ of $y_0$, such that $u(U)\subset K$ with $K$ a compact subset of $\Bbb{V}$,
then $u$ is smooth on a neighborhood of $y_0$. Especially if the image of $U$ under $u$ is contained in $S^n\backslash \ol{S}_+^{n-1}$, and
$$\sup_U (r^{-1}\circ u)<+\infty,$$
then the smoothness of $u$ near $y_0$ holds true.

\end{thm}

\begin{rem}

The main theorem in \cite{jxy} says a weakly harmonic map $u$ into $S^n$ is smooth near $y_0$ if the image of a neighborhood of $y_0$ is contained
in a compact subset of $S^n\backslash \ol{S}_+^{n-1}$. Our result is an improvement and a generalization of it.

\end{rem}

\begin{proof}

 We just give the proof of the first statement here, because the proof of the second one is quite similar.

By the definition of Riemannian manifolds, each point has a coordinate patch with induced metric. Hence without loss of generality
we can assume $U$ is a Euclidean ball centered at $y_0=0$ and of radius $R_0$ equipped with metric $g=g^{\a\be}dy^\a dy^\be$,
where $(y^1,\cdots,y^m)$ denotes Euclidean coordinate; and there exists two positive constants $\la$ and $\mu$, such that
$$\la^2|\xi|^2\leq g_{\a\be}(y)\xi^a\xi^\be\leq \mu^2|\xi|^2$$
for arbitrary $y\in U$ and $\xi\in \R^m$. By a standard scaling argument, we can assume $\la=1$ without loss of generality.

Let $d$ be the canonical Euclidean distance function, i.e.
$$d: (y_1,y_2)\in U\times U\mapsto |y_1-y_2|,$$
then obviously $d(\cdot,\cdot)\leq \rho(\cdot,\cdot)$, the distance function induced by $g$.

Denote $dy=dy^1\w\cdots\w dy^m$, then $*1=\sqrt{\det(g_{\a\be})}dy$ with
\begin{equation}\label{vol}
1\leq \sqrt{\det(g_{\a\be})}\leq \mu^m.
\end{equation}
It is well-know that $\n v=g^{\a\be}D^\a v D^\be v$, where $(g^{\a\be})$ is the inverse matrix of $(g_{\a\be})$, hence
\begin{equation}
|\n v|^2=g^{\a\be}D^\a v D^\be v\geq \mu^{-2}|Dv|^2.
\end{equation}
Recall that the classical Sobolev inequality says
\begin{equation}\label{sob}
\Big(\int_{B_R}v^{\f{m}{m-1}}dy\Big)^{\f{m-1}{m}}\leq C(m)\int_{B_R}|D v|dy
\end{equation}
for arbitrary nonnegative $C^1$-function $v$ on $B_R$ whose supporting set is contained in $B_R$. Then it could be derived from (\ref{vol})-(\ref{sob}) that
$$\aligned
&\Big(\int_{B_R}v^{\f{m}{m-1}}*1\Big)^{\f{m-1}{m}}\leq \Big(\mu^m\int_{B_R}v^{\f{m}{m-1}}dy\Big)^{\f{m-1}{m}}\\
\leq& \mu^{m-1}C(m)\int_{B_R}|D v|dy\leq \mu^{m}C(m)\int_{B_R}|\n v|*1.
\endaligned$$
By using H\"{o}lder inequality, it is easily-seen that for arbitrary $q\geq \f{m}{m-1}$,
\begin{equation}\label{sob1}
\Big(\int_{B_R} v^q*1\Big)^{\f{1}{q}}\leq \f{q(m-1)}{m}\mu^m C(m)\Big(\int_{B_R} |\n v|^{(\f{1}{m}+\f{1}{q})^{-1}}*1\Big)^{\f{1}{m}+\f{1}{q}}
\end{equation}
and moreover
\begin{equation}\label{sob2}
\Big(\aint{B_R}v^q\Big)^{\f{1}{q}}\leq \f{q(m-1)}{m}\mu^m C(m)V(R)^{\f{1}{m}}\Big(\aint{B_R}|\n v|^{(\f{1}{m}+\f{1}{q})^{-1}}\Big)^{\f{1}{m}+\f{1}{q}}
\end{equation}
It directly follows from (\ref{vol}) that $V(R)\leq \mu^m \om_m R^m$, with $\om_m$ the volume of $m$-dimensional Euclidean disk
equipped with canonical metric.
Hence (\ref{sob2}) enable us to choose
$$\nu=\left\{\begin{array}{cc} 4 & \text{if }m=2\\ \f{m}{m-2} & \text{if }m\geq 3\end{array}\right.$$
and
$$K_1=\f{2\nu(m-1)}{m}\mu^{m+1}\om_m^{\f{1}{m}}C(m)$$
to ensure (\ref{con2}) hold true.

By straightforward calculation similar to \cite{jxy} Section 6.1, one can make sure (\ref{double}) and (\ref{str-p}) hold true by putting
$K_2=(2\mu)^m$ and $K_3=\f{9}{4}\pi^{-2}\mu^{m+2}$. Hence Theorem \ref{t1} enable us to find two constant $C_0$ and $C_1$,
depending only on $m$, $\mu$ and $V$, such that $u(B_{R_1})$ is contained in a closed geodesic ball of radius $<\f{\pi}{2}$, if we denote
$$R_1=\f{1}{2}\exp\big(-C_1\exp (C_0M(R_0))\big)R_0\qquad \text{with }M(R_0):=\sup_{U}(r^{-1}\circ u).$$

Now we can proceed as in \cite{h-j-w} and \cite{jxy} to obtain estimates of the oscillation of $u$ and moreover the H\"{o}lder estimates
for $u$, which implies $u$ is H\"{o}lder continuous in a neighborhood of $y_0$. Finally $u$ has to be smooth near $y_0$
by the higher regularity results for harmonic maps.

\end{proof}

\bigskip

\Section{Curvature estimates for minimal hypersurfaces}{Curvature estimates for minimal hypersurfaces}

Let $M^m$ be an imbedded minimal hypersurface (not necessarily complete) in $(m+1)$-dimensional Euclidean space equipped with the induced Riemannian metric.
Denote the restriction of Euclidean distance function on $M$ by $d$:
$$(y_1,y_2)\in M\times M\mapsto |y_1-y_2|.$$
Then it is easily-seen that $d(y_1,y_2)\leq \rho(y_1,y_2)$, which are called the extrinsic and intrinsic function, respectively.
Since the inclusion map $i: M\ra \R^{m+1}$ is injective, the metric topology induced by $d$ coincides with initial topology of $M$.

Fix $y_0\in M$, denote by $B_R$ the intersection of $M$ and the Euclidean ball centered at $y_0$ and of radius $R$,
which is also the metric ball given by $d$. As shown in \cite{m-s}, for every nonnegative function $v$ of $C^1$-type which vanishes
outside a compact subset of $B_R$, the following Sobolev inequality
\begin{equation}
\Big(\int_{B_R} v^{\f{m}{m-1}}*1\Big)^{\f{m-1}{m}}\leq C(m)\int_{B_R} |\n v|*1
\end{equation}
holds. Then as in the proof of Theorem \ref{t2}, one can arrive at
\begin{equation}\label{sob3}
\Big(\aint{B_R}v^q\Big)^{\f{1}{q}}\leq \f{q(m-1)}{m} C(m)V(R)^{\f{1}{m}}\Big(\aint{B_R}|\n v|^{(\f{1}{m}+\f{1}{q})^{-1}}\Big)^{\f{1}{m}+\f{1}{q}}
\end{equation}
for arbitrary $q\geq \f{m}{m-1}$. Here and in the sequel, $V(R):=\text{Vol}(B_R)$.

The definition of $\om_m$ is similar to above. Given $y\in M$ and $R>0$, the volume density is define by
\begin{equation}
\mathcal{D}(y,R):=\f{V(y,R)}{\om_m R^m}
\end{equation}
The well-known monotonicity theorem tells us $\mathcal{D}(y,R)$ is nondecreasing in $R$
and $\lim_{R\ra 0^+}\mathcal{D}(y,R)=1$. Thus for arbitrary given $R_0>0$,
\begin{equation}\label{vol1}
\om_m R^m\leq V(R)\leq \mathcal{D}(R_0)\om_m R^m
\end{equation}
for all $R\in (0,R_0]$, where $\mathcal{D}(R_0)$ is the abbreviation of $\mathcal{D}(y_0,R_0)$.

By
(\ref{sob3}) and (\ref{vol1}), if we take
$$\nu:=\left\{\begin{array}{cc} 4 & \text{if }m=2\\ \f{m}{m-2} & \text{if }m\geq 3\end{array}\right.$$
and
$$K_1:=\f{2\nu(m-1)}{m}\mathcal{D}(R_0)^{\f{1}{m}}\om_m^{\f{1}{m}}C(m),$$
then (\ref{con2}) holds true for every $v\in H_0^{1,2}(B_R)$ with $R\in (0,R_0]$. (\ref{vol1}) also implies
so called 'doubling property' that
\begin{equation}\label{vol2}
V(R)\leq 2^m\mathcal{D}(R_0)V(\f{R}{2})
\end{equation}
for all $R\in (0,R_0]$. In other words, we can choose $K_2=2^m\mathcal{D}(R_0)$ so that (\ref{double}) holds.

Denote by $\mu_2(R)$ the second eigenvalue of $\De v+\mu v=0$ in $B_R$, where the normal derivative of $v$ vanishes
on the boundary of $B_R$. Then $\mu_2(R)$ is obviously continuous in $R$. We claim
\begin{equation}\sup_{R\in (0,R_0]}R^{-2}\mu_2(R)^{-1}<+\infty.
\end{equation}
To prove it, it is sufficient to show
\begin{equation}\label{eig3}
\limsup_{R\ra 0^+}R^{-2}\mu_2(R)^{-1}<+\infty.
\end{equation}
Every minimal hypersurface in $\R^{m+1}$ can be view as a
minimal graph over $\R^m$ locally. More precisely, by choosing suitable coordinate, we can assume
$y_0=0$ and $T_{y_0}M$ is orthogonal to the $(m+1)$-th coordinate vector, and there is a sufficiently small number
$R_-\leq R_0$, such that
$$B_{R_-}=\{y=(z,f(z)):z\in \Om\}$$
with a star-like domain $\Om$ in $\R^m$ and a function $f:\Om\ra \R$ satisfying $f(0)=0$, $|Df|(0)=0$, $|f|<1$ and $|Df|<1$ on $\Om$. Thereby $B_{R_-}$ is diffeomorphic to
Euclidean disk of radius $R_{-}$, and the diffeomorphism can be given by
$$\chi: y=(z,f(z))\mapsto \f{|y|}{|z|}z.$$
Obviously $\chi$ maps $B_R$ to an $m$-dimensional Euclidean ball of radius $R$ for each $R\leq R_-$. Hence the canonical Neumann-Poincar\'{e}
inequality on Euclidean spaces implies
\begin{equation}
\int_{B_R}|v-\bar{v}_{R}|^2*1\leq C R^2\int_{B_R}|\n v|^2*1
\end{equation}
for each $R\leq R_-$ and every function $v$ on $B_R$ of $H^{1,2}$-type, where $C$ is a positive constant which depends on $R_-$, $m$
and $M$, but not depends on $v$. (\ref{eig3}) is immediately followed from it. Now we denote
\begin{equation}
\La(R_0):=\sup_{R\in (0,R_0]}R^{-2}\mu_2(R)^{-1},
\end{equation}
then (\ref{poin2}) holds true for each $R\leq R_0$ if we choose $K_3=\f{9}{16}\La(R_0)$.

The Gauss map $\g:M\ra S^m$ is defined by
\begin{equation}
\g(y)=T_y M\in S^m
\end{equation}
via the parallel translation in $\R^m$ for every $y\in M$. Ruh-Vilms \cite{r-v} proved that $M$ has parallel mean curvature vector if and only if
 $\g$ is a harmonic map. Suppose the image of $B_{R_0}$ under Gauss map is contained in $S^m\backslash \ol{S}_+^{m-1}$,
and
\begin{equation}
M(R_0):=\sup_{B_{R_0}}(r^{-1}\circ \g)<+\infty.
\end{equation}
Then the image shrinking property of harmonic maps (Theorem \ref{t1}) allows us to find two positive constants $C_2$ and $C_3$,
depending only on $m$, $\mathcal{D}(R_0)$ and $\La(R_0)$; if we denote
\begin{equation}\label{R1}
R_1=\f{1}{2}\exp\big(-C_3\exp(C_2M(R_0))\big)R_0,
\end{equation}
then there is $\th_0\in \R$, such that every $y\in B_{R_1}$ satisfies
\begin{equation}\label{h}
(\g(y),x_0)\geq \f{1}{2}M(R_1)^{-1} \qquad \text{with }x_0=(\cos\th_0,\sin\th_0,0,\cdots,0).
\end{equation}

Denote
\begin{equation}
f=(\cdot,x_0)\circ \g,
\end{equation}
then $f$ is a positive function on $B_{R_1}$.
By (\ref{hess})
\begin{equation}
\Hess (\cdot,x_0)=-(\cdot,x_0)g_s
\end{equation}
with $g_s$ the canonical metric on $S^m$. Using composition formula we can get
\begin{equation}\label{la2}\aligned
\De f&=\Hess (\cdot,x_0)(\g_* e_\a,\g_* e_\a)+d\g(\tau(\g))\\
     &=-f|d\g|^2\endaligned
\end{equation}
Denote by $B$ the second fundamental form of $M$ in $\R^{m+1}$; as shown in \cite{x} Chap. 3. \S 3.1, the energy density
of Gauss map
\begin{equation}\label{energy}
E(\g)=\f{1}{2}|d\g|^2=\f{1}{2}|B|^2.
\end{equation}
Substituting (\ref{energy}) into (\ref{la2}) yields
\begin{equation}\label{la3}
\De f=-|B|^2f.
\end{equation}
Let
\begin{equation}
h:=f^{-1}=(\cdot,x_0)^{-1}\circ \g,
\end{equation}
then from (\ref{la3}) we arrive at
\begin{equation}\label{la4}
\De h=|B|^2h+2h^{-1}|\n h|^2.
\end{equation}

The following Simons' identity \cite{si} is well-known
\begin{equation}\label{si}
\De |B|^2=-2|B|^4+2|\n B|^2.
\end{equation}
With the aid of Codazzi equations, Schoen-Simon-Yau \cite{s-s-y} get a Kato-type inequality as follows
\begin{equation}\label{kato}
|\n B|^2\geq \big(1+\f{2}{m}\big)\big|\n |B|\big|^2.
\end{equation}
And it follows from (\ref{si}) and (\ref{kato}) that
\begin{equation}\label{la5}
\De |B|^2\geq -2|B|^4+2\big(1+\f{2}{m}\big)\big|\n|B|\big|^2.
\end{equation}
Based on (\ref{la4}) and (\ref{la5}), $\De (|B|^ph^q)$ can be easily calculated for arbitrary $p,q>0$; by choosing
suitable $p,q$, one can
proceeded as in \cite{e-h} to get
\begin{equation}
\De(|B|^p h^p)\geq 0
\end{equation}
for arbitrary $p\geq \f{m-2}{2}$.
The mean value inequality on minimal submanifolds (see \cite{c-l-y}, \cite{n}) can be applied to get
\begin{equation}\label{cur1}
|B|^ph^p(y_0)\leq C(m)V(R)^{-\f{1}{2}}\Big(\int_{B_R}|B|^{2p}h^{2p}*1\Big)^{\f{1}{2}}
\end{equation}
for arbitrary $R\leq R_1$.
Again using the inequality for $\De (|B|^ph^q)$, one can get the following estimate as in \cite{e-h}:
\begin{equation}
\int_{B_{R_1}}|B|^{2p}h^{2p}\eta^{2p}*1\leq C(p)\int_{B_{R_1}}h^{2p}|\n \eta|^{2p}*1.
\end{equation}
Here $p\geq \max\{3,m-1\}$ and $\eta$ can be taken by any smooth function which vanishes outside a compact subset of $B_{R_1}$.
Now we choose $\eta$ to be standard cut-off function satisfying $\text{supp}\ \eta\subset B_{R_1}$, $\eta\equiv 1$ on $B_{\f{R_1}{2}}$
and $|\n \eta|\leq c_0R_1^{-1}$ and we get
\begin{equation}\label{cur2}
\int_{B_{\f{R_1}{2}}}|B|^{2p}h^{2p}*1\leq C(p)c_0^{2p}R_1^{-2p}V(R_1)\sup_{B_{R_1}}h^{2p}.
\end{equation}
Substituting (\ref{cur2}) into (\ref{cur1}) implies
\begin{equation}\label{cur3}
|B|^p h^p(y_0)\leq C(m,p)\Bigg(\f{V(R_1)}{V(\f{R_1}{2})}\Bigg)^{\f{1}{2}}R_1^{-p}\sup_{B_{R_1}}h^p.
\end{equation}
Then by combining with (\ref{cur3}), (\ref{vol2}), (\ref{R1}) and (\ref{h}) we obtain a prior curvature estimate
as follows:

\begin{thm}\label{t3}

Let $M^m\subset \R^{m+1}$ be an imbedded minimal hypersurface, $y_0$ be an arbitrary point in $M$. Denote
$$B_R=\{y\in M: |y-y_0|\}<R.$$
If there is $R_0>0$, such that the Gauss image of $B_{R_0}$ is contained in $S^m\backslash\ol{S}_+^{m-1}$,
and
$$\sup_{B_{R_0}}(r^{-1}\circ \g)<+\infty,$$
the we have the following estimate
\begin{equation}\label{cur4}
|B|(y_0)\leq C_4R_0^{-1}\exp\big(C_3\exp(C_2\sup_{B_{R_0}}(r^{-1}\circ \g))\big).
\end{equation}
Here $C_2,C_3,C_4$ are positive constants only depending on $m, \mathcal{D}(R_0)$ and $\La(R_0)$,
where $\mathcal{D}(R_0):=\f{V(R_0)}{\om_m R_0^m}$ and $\La(R_0):=\sup_{R\in (0,R_0]}R^{-2}\mu_2(R)^{-1}$.
\end{thm}

\begin{rem}

If the Gauss image of $B_{R_0}$ is contained in $K\subset \Bbb{V}$, where $\Bbb{V}$ is a simply connected subset of
$S^m\backslash S^{m-2}$, then $\sup_{B_{R_0}}(r^{-1}\circ \g)<+\infty$. From the image shrinking property, we can proceed as above and get
\begin{equation}
|B|(y_0)\leq C_5R_0^{-1}
\end{equation}
with a positive constant $C_5$ depending only on $m, \mathcal{D}(R_0),\La(R_0)$ and $K$.

\end{rem}

Now we additionally assume $M$ to be complete and the Gauss image of $M$ is contained in
$S^m\backslash \ol{S}_+^{m-1}$. If there exists $y_1\in M$ and a positive constant $C$ such that
\begin{equation}
\mathcal{D}(y_1,R)\leq C\qquad \text{for every }R<+\infty,
\end{equation}
then we say $M$ has Euclidean volume growth. For arbitrary $y\in M$, if we denote $d=d(y,y_1)$, then
$$V(y,R)\leq V(y_1,R+d)\leq C\om_m(R+d)^m$$
and moreover
$$\mathcal{D}(y,R)\leq C\Big(\f{R+d}{R}\Big)^m$$
Letting $R\ra +\infty$ implies $\lim_{R\ra +\infty}\mathcal{D}(y,R)\leq C$, then monotonicity theorem tells us
\begin{equation}
\mathcal{D}(y,R)\leq C\qquad \text{for every }y\in M\text{ and }R<+\infty.
\end{equation}
Using image shrinking property and above curvature estimates, one can get a Bernstein-type theorem as follows.

\begin{thm}\label{t4}
Let $M^m\subset \R^{m+1}$ be an imbedded complete minimal hypersurface with Euclidean volume growth, and the image under Gauss
map is contained in $S^m\backslash \ol{S}_+^{m-1}$. Assume there is $y_0\in M$,
such that $\lim_{R\ra +\infty}\La(R)<+\infty$, and
\begin{equation}\label{order}
\sup_{B_R}(r^{-1}\circ \g)=o(\log\log R).
\end{equation}
Then $M$ has to be an affine linear subspace.

\end{thm}

\begin{proof}

Denote $\Th=\th\circ \g$ and $M(R)=\sup_{B_R}(r^{-1}\circ\g)$. Since $\g$ is harmonic, Theorem \ref{t1} enable us to find two positive constants
$C_0$ and $C_1$ depending only on $m$, $\lim_{R\ra +\infty}\mathcal{D}(R)$ and $\lim_{R\ra +\infty}\La(R)$; if we denote
$$R'=\f{1}{2}\exp\big(-C_1\exp(C_0M(R))\big)R,$$
then there is $\th_0=\th_0(R)\in [\f{\pi}{3},\f{5\pi}{3}]$, such that
\begin{equation}\label{th2}
\big|\Th(y)-\th_0(R)\big|\leq \f{\pi}{3}\qquad \text{for every }y\in B_{R'}.
\end{equation}

By the compactness of $[\f{\pi}{3},\f{5\pi}{3}]$, there is an monotonicity increasing sequence $\{R_j:j\in \Bbb{Z}^+\}$ satisfying
$\lim_{j\ra \infty}R_j=+\infty$ and $\lim_{j\ra \infty}\th(R_j)=\th_\infty\in [\f{\pi}{3},\f{5\pi}{3}]$.
Denote
$$R'_j=\f{1}{2}\exp\big(-C_1\exp(C_0M(R_j))\big)R_j.$$
(\ref{order}) implies for arbitrary $\ep>0$, there is $k\in \Bbb{Z}^+$, such that for every $j\geq k$, $M(R_j)\leq \ep\log\log R_j$, hence
$$R'_j\geq \f{R_j}{2\exp\big(C_1(\log R_j)^{C_0\ep}\big)}.$$
When $\ep$ is sufficiently small and $R_j$ is sufficiently large, one can have $C_1(\log R_j)^{C_0\ep}\leq \f{1}{2}\log R_j$, which implies
$$R'_j\geq \f{1}{2}R_j^{\f{1}{2}}$$
and hence $\lim_{j\ra \infty}R'_j=+\infty$.

Hence for arbitrary $y\in M$, we can find $l\in \Bbb{Z}^+$, such that $d(y_0,y)\leq R'_j$ whenever $j\geq l$. (\ref{th2}) tells us
$$\big|\Th(y)-\th_0(R_j)\big|\leq \f{\pi}{3}.$$
Letting $j\ra \infty$ in above inequality we arrive at
\begin{equation}
\big|\Th(y)-\th_\infty\big|\leq \f{\pi}{3}\qquad \text{for every }y\in M.
\end{equation}
It implies the Gauss image of $M$ is contained in an open hemisphere centered at
 $x_0:=(\cos\th_\infty,\sin\th_\infty,0,\cdots,0)$.

Let $h:=(\cdot,x_0)^{-1}\circ \g$, then for arbitrary $y\in M$, similarly to above we can arrive at the following estimate
\begin{equation*}\aligned
|B|^ph^p(y)&\leq C(m,p)\Bigg(\f{V(R)}{V(\f{R}{2})}\Bigg)^{\f{1}{2}}R^{-p}\sup_{B_R(y)}h^p\\
           &\leq c_3 R^{-p}\sup_{B_R(y)}(r^{-1}\circ \g)^p;
           \endaligned
\end{equation*}
i.e.
\begin{equation}
|B|(y)\leq c_3^{\f{1}{p}}R^{-1}\sup_{B_{R+d}}(r^{-1}\circ \g)
\end{equation}
with $d=d(y,y_0)$ and $c_3$ a positive constant depending only on $m,p$ and $\lim_{R\ra +\infty}\mathcal{D}(R)$.
Letting $R\ra +\infty$ forces $|B|(y)=0$. Therefore $M$ has to be flat.

\end{proof}

\begin{rem}

$\lim_{R\ra +\infty}\La(R)<+\infty$ is equivalent to say that Neumann-Poincar\'{e} inequality
\begin{equation}\label{poin}
\int_{B_R(y_0)}|v-\bar{v}_R|^2*1\leq CR^2\int_{B_R(y_0)}|\n v|^2*1
\end{equation}
holds for every $R$ and arbitrary $C^1$-function $v$ on $B_R(y_0)$ with positive constant $C$.
Please note that Bernstein type theorem in \cite{jxy} requires (\ref{poin}) holds for every $y_0\in M$ with a uniform constant $C$,
so our assumption on $M$ is weaker.

Moreover, in \cite{jxy}, one assume that the Gauss image of $M$ omits a neighborhood of $\ol{S}_+^{n-1}$, which implies $r^{-1}\circ \g$ is
bounded and $\th\circ \g$ is contained in a closed interval $\subset(0,2\pi)$. In contrast, here $\th\circ\g$ is allowed to converge to $0$ or $2\pi$
when $y$ diverges to infinity at arbitrary speed and meanwhile $r^{-1}\circ g$ can increase to $+\infty$ in a controlled manner.
It is an improvement.

\end{rem}

\begin{rem}

In Theorem \ref{t4}, if we replace the condition on Gauss image of $M$ by assuming $\g(M)\subset K\subset \Bbb{V}$ with
$\Bbb{V}$ a simply connected subset of $S^m\backslash S^{m-2}$. Then again based on image shrinking property
we can get the corresponding Bernstein type result. We note that it is a generalization of Bernstein type theorem in \cite{jxy}.

\end{rem}

Especially if $M$ is area-minimizing, the Neumann-Poincar\'{e} inequality (\ref{poin}) holds for every $y_0\in M$ and $R$ with a uniform
constant $C$ only depending on $m$; the result is due to Bombieri-Giusti \cite{b-g}. Meanwhile, for volume of extrinsic balls we have (see \cite{b-g})
\begin{equation}
\text{Vol}(B_R(y))\leq \f{m+1}{2}\om_{m+1}R^m\qquad \text{for every }y\in M.
\end{equation}
Therefore $M$ satisfies local DSVP-condition with respect to arbitrary $y\in M$ with constants $K_1,K_2,K_3$ and $\nu$ which all depend only on $m$
and furthermore (\ref{sh2}) holds with positive constants $C_0$ and $c_1$ only depending on $m$. Starting from (\ref{sh2}) one can derive
another Bernstein type theorem as follows.

\begin{thm}\label{t5}

Let $M^m\subset \R^{m+1}$ be a complete imbedded area-minimizing hypersurface. There is a positive constant
$\ep=\ep(m)$, if the Gauss image of $M$ is contained in $S^m\backslash \ol{S}_+^{m-1}$,
and
\begin{equation}\label{growth}
\sup_{B_R(y_0)}(r^{-1}\circ \g)\leq \ep(m)\log\log R
\end{equation}
for a point $y_0\in M$ and every $R\geq R_->0$, then $M$ has to be an affine linear space.

\end{thm}

\begin{proof}

The denotation of $C_0$ and $c_1$ is same as above. Let
\begin{equation}\ep=C_0^{-1},
\end{equation}
then
\begin{equation}\label{sh5}
\aligned
\int_{R_-}^{+\infty} R^{-1}\exp(-C_0M(R))dR&\geq \int_{R_-}^{+\infty} R^{-1}(\log R)^{-C_0\ep}dR\\
                                     &=\int_{R_-}^{+\infty} (\log R)^{-1}d\log R\\
                                     &=\log\log R|_{R_-}^{+\infty}=+\infty.
\endaligned
\end{equation}
By (\ref{sh2}) and (\ref{sh5}), for arbitrary $R\geq R_-$, one can take $R_0$ large enough, such that $\log \text{osc}_{B_R}\Th-\log\text{osc}_{B_{R_0}}\Th\leq -\log 3$;
in conjunction with $\text{osc}_{B_{R_0}}\Th\leq 2\pi$, we can find $\th_0(R)\in [\f{\pi}{3},\f{5\pi}{3}]$, such that
$$\Th|_{B_R}\in [\th_0(R)-\f{\pi}{3},\th_0(R)+\f{\pi}{3}].$$
The compactness of $[\f{\pi}{3},\f{5\pi}{3}]$ enable us to find a strictly increasing sequence $\{R_j:j\in \Bbb{Z}^+\}$ converging to $+\infty$
and satisfying $\lim_{j\ra \infty}\th_0(R_j)=\th_\infty\in [\f{\pi}{3},\f{5\pi}{3}]$. Similarly to above we can derive
$|\Th(y)-\th_\infty|\leq \f{\pi}{3}$ for every $y\in M$. It follows that the Gauss image of $M$ is contained in a closed subset
of open hemisphere. Finally Ecker-Huisken estimate \cite{e-h} implies $M$ has to be affine linear.

\end{proof}

\begin{rem}

Similarly, if the Gauss image of $M$ is contained in $K\subset \Bbb{V}$ with $\Bbb{V}$ a simply connected subset of
$S^m\backslash S^{m-2}$, then our conclusion still holds true. Thereby we not only improve, but also generalize the results of
Theorem 6.6 in \cite{jxy}.

\end{rem}

\begin{cor}\label{cor2}

Let $f$ be an entire solution of the minimal surface equation
\begin{equation}
\sum_{i=1}^m D^i\Big(\f{D^i f}{\sqrt{1+|D f|^2}}\Big)=0
\end{equation}
in $\R^m$ with $f(0)=0$.
There is a positive constant $\de=\de(m)$, if
\begin{equation}\label{graph2}
\Big(\sum_{i=1}^{m-1}(D^i f)^2\Big)^{\f{1}{2}}\leq \de(m)\log\log \big(f^2+x^2\big)^{\f{1}{2}}
\end{equation}
holds for arbitrary $|x|\geq R_->0$, then $f$ has be to affine linear.

\end{cor}

\begin{proof}

Under the assumptions, $M=\{(x,f(x)):x\in \R^n\}$ is an entire minimal graph, which is area-minimizing from classical minimal surface theory.
For every $x\in \R^n$,
$$\g(x,f(x))=(1+|Df|^2)^{-\f{1}{2}}(-D^1 f,-D^2 f,\cdots,-D^m f,1).$$
Thus the Gauss image of $M$ is contained in an open hemisphere. Now we define $\pi: S^m\ra \overline{\Bbb{D}}$ by
$$(x_1,\cdots,x_{m+1})\ra (x_m,x_{m+1}).$$
Then
\begin{equation}\label{graph}
r^{-2}\circ \g=\f{1+|Df|^2}{1+(D^m f)^2}=1+\f{\sum_{i=1}^{m-1}(D^i f)^2}{1+(D^m f)^2}\leq 1+\sum_{i=1}^{m-1}(D^i f)^2.
\end{equation}
$f(0)=0$ implies $0\in M$; take $y_0=0$, then for arbitrary $y=(x,f(x))\in M$, $y\in B_R(y_0)$ if and only if $\big(f^2+x^2\big)^{\f{1}{2}}<R$.
Now we choose $\de(m)=\f{1}{2}\ep(m)$, where the definition of $\ep(m)$ is same as in Theorem \ref{t5}, then (\ref{graph}) and (\ref{graph2}) imply
(\ref{growth}) when $R$ is large enough. And our conclusion is immediately followed from Theorem \ref{t5}.

\end{proof}

\begin{rem}

By Corollary \ref{cor2}, if $D^1 f,\cdots,D^{m-1} f$ are uniformly bounded in $\R^m$, then $f$ is affine linear. So Corollary \ref{cor2} is an improvement
of Theorem 8 in \cite{b-g}. It is also comparable with Ecker-Huisken's results (see \cite{e-h}).

\end{rem}
\bigskip

\Section{Appendix}{Appendix}

As shown in \cite{G}, any harmonic map from a compact Riemannian manifold into a convex supporting set $\Bbb{V}$
has to be constant. Especially, arbitrary closed geodesic can be viewed as a harmonic map from $S^1$ into $\Bbb{V}$,
hence every convex supporting set cannot contain any closed geodesic.

Conversely, if a subset $\Bbb{V}$ of a Riemannian manifold $(M,g)$ contains no closed geodesic, does $\Bbb{V}$
have to be convex supporting? Unfortunately the answer is 'no'. The following is a counterexample. Let
$M=S^2$, $S^1$ be the equator and
\begin{equation*}
A:=\big\{(\cos\th,\sin\th,0)\in S^1:\th\in [0,\f{\pi}{3}]\cup [\f{2\pi}{3},\pi]\cup [\f{4\pi}{3},\f{5\pi}{3}]\big\}.
\end{equation*}
Noting that every great circle $\mc{C}$ intersects $S^1$ at least 2 antipodal points, we have $\mc{C}\cap A\neq \emptyset$
and hence $\Bbb{V}:=S^2\backslash A$ contains no closed geodesic. On the other hand, we show $\Bbb{V}$ is not a convex supporting set.
Let $R_{\f{\pi}{3}}$ be $\f{\pi}{3}$-rotation around $z$-axis. Choose compact subset $K$, such that $N,S\in K$ and $R_{\f{\pi}{3}}(K)=K$
(if not, replace $K$ instead of $K\cup R_{\f{\pi}{3}}(K)\cup R_{\f{2\pi}{3}}(K)$). Assume $f$ is a $C^2$-convex function
on $K$, then $f\circ R_{\f{\pi}{3}}$ and $f\circ R_{\f{2\pi}{3}}$ are also convex, hence
\begin{equation*}
h:=f+f\circ R_{\f{\pi}{3}}+f\circ R_{\f{2\pi}{3}}
\end{equation*}
is a convex function on $K$ which is invariant under $R_{\f{\pi}{3}}$. Since $N$ and $S$ are fixed points
under $R_{\f{\pi}{3}}$, we have
$$\big(R_{\f{\pi}{3}}\big)_*(\n h)=\n h$$
at $N$ and $S$. Thus $\n h(N)=\n h(S)=0$, which means $h$ has 2 critical points in $K$.
It contradicts to the convexity of $h$. Therefore $\Bbb{V}$ cannot be convex supporting.

Moreover, $S^n\backslash \ol{S}_+^{n-1}$ is the unique maximal convex supporting set of $S^n$ that contains
the upper hemisphere and lower hemisphere.

\begin{pro}

Let $\Bbb{V}$ be an open and connected convex supporting set of $S^n$, if $S_+^n\cup S_-^n\subset \Bbb{V}$, then
$\Bbb{V}\subset S^n\backslash \overline{S}_+^{n-1}$.

\end{pro}

\begin{proof}

Let $\{e_1,\cdots,e_{n+1}\}$ be an orthornormal basis of $\R^{n+1}$, and
$$S_+^n=\{x\in S^n:(x,e_1)>0\},\qquad S_-^n=\{x\in S^n:(x,e_1)<0\}.$$
By the definition of convex supporting sets, for arbitrary compact subset $K\subset \Bbb{V}$, we can find a strictly convex function
$f$ on $K$. Now we choose a family of compact sets $\{K_i\subset V:i=1,2,\cdots\}$,
such that $K_i\subset K_j$ for arbitrary $i<j$, $\Bbb{V}=\bigcup_{i=1}^\infty K_i$, and
each $K_i$ satisfies the following 2 conditions: (I) $K_i$ is invariant under the reflection
with respect to the hyperplane $(\cdot,e_1)=0$; (II) For arbitrary $x\in K_i$ satisfying $(x,e_1)=0$, the geodesic from $e_1$
to $-e_1$ which goes through $x$ is contained in $K_i$. We denote by $f_i$ the convex function on $K_i$.

Now we denote by $\psi$ the reflection with respect to $(\cdot,e_1)=0$, then obviously $\psi$ is an isometry and
hence $f\circ \psi$ is also strictly convex. Let
$$h_i=\f{1}{2}(f_i+f_i\circ \psi),$$
then $h_i$ is a strictly convex function which is invariant under $\psi$, in particular $h_i(e_1)=h_i(-e_1)$.

If $\n h_i=0$ at $e_1$,
then $h_i\circ \psi=h_i$ implies $\n h_i=0$ at $-e_1$; it means that $h_i$ has 2 critical points in $K$, which contradict to the convexity
of $h$. Hence $\n h_i\neq 0$. Denote $v_i=\f{\n h_i}{|\n h_i|}$. Now we claim
$$H_i=\{x\in S^n: (x,e_1)=0,(x,v_i)>0\}$$
satisfies $H_i\cap K_i=\emptyset$. We prove it by Reductio ad absurdum. Assume
$x\in H_i\cap K_i$, then by Condition (II) there is a geodesic $\g$ lying in $K$ which connects $e_1$ and $-e_1$
and goes through $x\in K_i\cap H_i$; hence $\lan \dot{\g},v_i\ran>0$ and moreover $\f{d}{dt}\Big|_{t=0}(h_i\circ \g)>0$; the convexity of $h_i$
tells us $h_i\circ \g$ is a strictly increasing function, which contradict to $h_i(e_1)=h_i(-e_1)$.

The compactness of $T_{e_1}S^n$ enable us to find a subsequence of $\{v_i:i=1,2,\cdots\}$ converging to a unit vector in
$T_{e_1}S^n$. Without loss of generality one can assume
$$e_2=\lim_{i\ra \infty}v_i.$$
Denote
$$S_+^{n-1}=\{x\in S^n: (x,e_1)=0,(x,e_2)>0\}.$$
Then for every $x\in S_+^{n-1}$, we can find $k\in \Bbb{Z}^+$, such that $x\in H_i$ for every $i\geq k$. Therefore
$x\notin K_i$ and furthermore $x\notin \bigcup_{i=k}^\infty K_i=\Bbb{V}$. It follows $S_+^n\cap \Bbb{V}=\emptyset$, and the conclusion
immediately follows.

\end{proof}

\bibliographystyle{amsplain}

\begin{thebibliography}{10}

\bibitem{bm}M. Biroli, U. Mosco:
Sobolev inequalities for Dirichlet forms on homogeneous spaces. Boundary value problems for partial differential equations and applications
essais for 70th birthday of E. Magenes, Paris, 1993.

\bibitem{b-g}E. Bombieri, E. Giusti:
Harnack's inequality for elliptic differential equations on minimal
surfaces. Invent. Math. {\bf 15}(1972), 24-46.

\bibitem{c-l-y}S. Y. Cheng, P. Li  and S. T. Yau:
Heat equations on minimal submanifols and their applications. Amer.
J. Math. {\bf  103}(1981), 1021-1063.

\bibitem{cm} T. Colding and W. Minicozzi:
Harmonic functions on manifolds.
Ann. of Math. {\bf 146(3)}(1997), 725-747.

\bibitem{e-h} K. Ecker  and  G. Huisken:
A Bernstein result for minimal graphs of controlled growth. J. Diff.
Geom. {\bf 31}(1990), 397-400.

\bibitem{g-g} M. Giaquinta and E. Giusti:
On the regularity of the minima of variational integrals. Acta. Math.
{\bf 148}(1982), 31-46.

\bibitem{g-h} M. Giaquinta and S. Hildebrandt:
A prior estimeats for harmonic mappings. J. Reine Angew. Math. {\bf 336}
(1982),  124-164.


\bibitem{G} W. B. Gordon:
Convex functions and harmonic maps. Proc. Amer. Math. Soc. {\bf 33(2)}(1972),
433-437.

\bibitem{g-w} M. Gr\"{u}ter and K. Widman:
The Green function for uniformly elliptic equations. Manu. Math. {\bf 37}
(1982), 303-342.

\bibitem{g-j} R. Gulliver and J. Jost:
Harmonic maps which solve a free-boundary problem. J. Reine Angew. Math.
{\bf 381}(1987), 61-89.

\bibitem{h-j-w} S. Hildebrandt, J. Jost and K. Widman:
Harmonic mappings and minimal submanifolds. Invent. Math. {\bf
62}(1980), 269-298.

\bibitem{h-k-w} S. Hildebrandt, H. Kaul and K. Widman:
An existence theorem for harmonic mappings of Riemannian manifolds,
Acta. Math. {\bf 138}(1977), 1-16.

\bibitem{j} J. Jost:
Generalized Dirichlet forms and harmonic maps. Calc. Var. PDE {\bf
5}(1997), 1-19.

\bibitem{J} J. Jost:
Riemannian geometry and geometric analysis. 4th Edition. Springer, 2008.

\bibitem{jxy} J. Jost, Y. L. Xin and Ling Yang:
The regularity of harmonic maps into spheres and applications to
Bernstein problems.  J. Diff. Geom. {\bf 90} (2012), 131-176.

\bibitem{m-s} J. Michael and L. Simon:
Sobolev and mean-value inequalities on generalized submanifolds of $\R^n$.
Comm. Pure Appl. Math. {\bf 26}(1973), 361-379.

\bibitem{m} J. Moser:
On Harnack's theorem for elliptic differential equations. Comm. Pure
Appl. Math. {\bf 14}(1961), 577-591.

\bibitem{n} Lei Ni:  Gap theorems for minimal submanifolds in
$\ir{n+1}$. Comm. Analy.  Geom. {\bf 9 (3)}(2001), 641-656.

\bibitem{r-v} E. A. Ruh and J. Vilms:
The tension field of Gauss map. Trans. Amer. Math. {\bf 149}(1970),
569-573.

\bibitem{sc} L. Saloff-Coste:
A note on Poincar\'{e}, Sobolev and Harnack inequalities. Int. Math. Res. Notices {\bf 2}(1992),
27-38.

\bibitem{s-s-y} R. Schoen, L. Simon and S. T. Yau:
Curvature estimates for minimal hypersurfaces. Acta Math. {\bf 134}
(1975), 275-288.

\bibitem{si} J. Simons:
Minimal varieties in Riemannian manifolds. Ann. Math. {\bf 88}
(1968), 62-105.

\bibitem{so} B. Solomon:
Harmonic maps to spheres. J. Diff. Geom. {\bf 21} (1985), 151-162.

\bibitem{x} Y. L. Xin: Geometry of harmonic maps.
Birkh\"auser PNLDE 23, 1996.
\end{thebibliography}

\end{document}